\documentclass[10pt,a4paper]{article}
\usepackage[utf8]{inputenc}
\usepackage{amsfonts}
\usepackage{amssymb}
\usepackage{amsmath}
\usepackage{color}
\usepackage{graphicx}
\usepackage{placeins}
\usepackage{amsthm}
\usepackage{enumerate}
\PassOptionsToPackage{normalem}{ulem}
\usepackage{ulem}
\usepackage[unicode=true,colorlinks,citecolor=linkcolor,linkcolor=linkcolor,urlcolor=urlcolor]{hyperref}
\usepackage[left=3cm,right=3cm,top=2.5cm,bottom=2.5cm]{geometry}
\usepackage{array}
\usepackage{cases}

\usepackage{color}
\usepackage{dsfont}

\theoremstyle{plain}
\newtheorem{thm}{\protect Theorem}[section]
\newtheorem{lemma}[thm]{\protect Lemma}

\newtheorem{remark}[thm]{\protect Remark}

\newtheorem{corollary}[thm]{\protect Corollary}
  
\newcommand{\card}{\operatorname{Card}}
\newcommand{\supp}{\operatorname{Supp}}
\newcommand{\diam}{\operatorname{Diam}}

\definecolor{linkcolor}{rgb}{0,0,0.502}
\definecolor{urlcolor}{rgb}{1,0,0}

\begin{document}

\title{Uniform decomposition of probability measures: \\
quantization, classification, rate of convergence.}

\author{Julien Chevallier\footnote{{\sf{e-mail: \href{mailto:julien.chevallier1@univ-grenoble-alpes.fr}{julien.chevallier1@univ-grenoble-alpes.fr}}}}\\
Univ. Grenoble Alpes, CNRS, LJK, 38000 Grenoble, France
}

\date{}

\maketitle

\begin{abstract}
The study of finite approximations of probability measures has a long history. In (Xu and Berger, 2017), the authors focus on constrained finite approximations and, in particular, uniform ones in dimension $d=1$. The present paper gives an elementary construction of a uniform decomposition of probability measures in dimension $d\geq 1$. This decomposition is then used to give upper-bounds on the rate of convergence of the optimal uniform approximation error. These bounds appear to be the generalization of the ones obtained in (Xu and Berger, 2017) and to be sharp for generic probability measures. 
\end{abstract}


\section{Introduction}

Finding a good finite decomposition of a given probability measure $\rho$ on $\mathbb{R}^{d}$ is an extensively studied problem. \emph{Quantization} is concerned with the best finitely supported approximation of a probability measure (empirical measures being especially studied for \emph{classification}). The origins come from signal processing (optimal signal transmission through discretization) \cite{bennett1948spectra} but the range of application widened since then (pattern recognition \cite{faundez2011efficient}, numerical analysis \cite{pages2015introduction}, economics \cite{pages2010optimal}). The goodness of the approximation is usually measured in terms of an $L^{p}$-Wasserstein distance $W_{p}$ and numerous results are concerned with the rate of convergence of $e_{p,n}(\rho):=\inf W_{p}(\rho^{(n)},\rho)$ to 0 where the infimum is taken with respect to the set of measures $\rho^{(n)}$ supported by at most $n$ atoms \cite{graf2007foundations}.

\emph{Random empirical quantization} has recently attracted much attention \cite{boissard2014mean,bolley2007quantitative,fournier2014rate} in particular for its application to mean-field interacting particle systems. In that case, the approximating measure is $R^{(n)}=n^{-1}\sum_{k=1}^{n} \delta_{X_{k}}$ where the $X_{k}$'s are i.i.d. random variables distributed according to $\rho$ and the main results are concerned with rate of convergence of $\mathbb{E}\left[W_{p}(R^{(n)},\rho) \right]$ or concentration inequalities of the random variable $W_{p}(R^{(n)},\rho)$.\\

In that context, when the approximating measure is $\mu^{(n)}=n^{-1}\sum_{k=1}^{n} \delta_{x_{k}}$ with deterministic $x_{k}$'s, we use the term \emph{deterministic empirical quantization}. This kind of approximation is used for instance when considering mean-field limits with spatial covariates used to weight the interactions between particles \cite{chevallier2016spatial}. The case of dimension $d=1$ is extensively adressed in \cite{xu2017best} (the study highly relies on the connection between Wasserstein distances and the quantile function which is specific to $d=1$). The aim of the present paper is to generalize some of the results stated in \cite{xu2017best} to the general case $d\geq 1$. The main result gives sharp bounds on the rate of convergence of $\tilde{e}_{p,n}(\rho):=\inf W_{p}(\mu^{(n)},\rho)$ to 0 where the infimum is taken with respect to the set of deterministic empirical measures $\mu^{(n)}$ supported by $n$ atoms. The rate of convergence depends on the dimension $d$ and the order $p$ and shows a transition: it is either the same as for standard quantization (when Lebesgue measure is harder to approximate) or strictly worse (when disconnected measures are harder to approximate).

The paper is organized as follows. Definitions and notation are given in Section \ref{sec:def:not:results} with a list of previous results found in the literature. Then, Section \ref{sec:main:result} contains an elementary uniform decomposition of probability measures (Theorem \ref{thm:decomposition}) which is used to obtain upper-bounds on deterministic empirical quantization rates (Theorem \ref{thm:rate:bounded:support}) and uniform classification rates (Corollary \ref{cor:classification}).

\section{Notation and previous results}
\label{sec:def:not:results}

The space $\mathbb{R}^{d}$ is equipped with the maximum norm $||.||$ and the balls centered at $0$ are denoted by $\mathcal{B}_{r}:=B(0,r)=[-r,r]^{d}$ for all $r\geq 0$. The diameter of a subset $A$ of $\mathbb{R}^{d}$ is denoted by $\diam(A):=\sup_{x,y\in A} ||x-y||$. The space of every Borel measures (\emph{resp.} probability measures) on $\mathbb{R}^{d}$ is denoted by $\mathcal{M}(\mathbb{R}^{d})$ (\emph{resp.} $\mathcal{P}(\mathbb{R}^{d})$). For $\nu$ in $\mathcal{M}(\mathbb{R}^{d})$, $\supp(\nu)$ and $|\nu|:=\nu(\mathbb{R}^{d})$ respectively denote the support and the mass of the measure $\nu$. For a collection of $n$ positions $x_{1},\dots,x_{n}$ in $\mathbb{R}^{d}$, we denote its associated empirical measure by $\mu^{(n)}:=n^{-1}\sum_{k=1}^{n}\delta_{x_{k}}$. 

For every $p\geq 1$, the set of probability measures $\rho$ such that $\int ||x||^{p} \rho(dx)<+\infty$ is denoted by $\mathcal{P}_{p}$. Then the Wasserstein distance of order $p$ is denoted by $W_{p}$ and defined by, for all $\rho$ and $\mu$ in $\mathcal{P}_{p}$,
\begin{equation*}
W_{p}(\rho,\mu):= \left(\inf_{\pi} \int_{(\mathbb{R}^{d})^{2}} ||x-y||^{p} \pi(dx,dy)\right)^{1/p},
\end{equation*}
where the infimum is taken with respect to every couplings $\pi$ of the two measures $\rho$ and $\mu$.

\subsection{State of the art}

Given $\rho$ in $\mathcal{P}_{p}$ the \emph{optimal quantization error} of order $p$ is defined as 
\begin{equation*}
e_{p,n}(\rho):=inf_{\rho^{(n)}} W_{p}(\rho^{(n)},\rho),
\end{equation*}
where the infimum is taken with respect to the set of measures $\rho^{(n)}$ supported by at most $n$ atoms. The literature dealing with the rate of convergence of $e_{p,n}(\rho)$ to 0 is extensive \cite{graf2007foundations,kloeckner2012approximation,zador1963development,zador1982asymptotic}. One of the most celebrated result is due to Zador \cite[Theorem 6.2]{graf2007foundations}. A consequence says that if $\rho$ is in $\mathcal{P}_{q}$ for some $q>p$ and admits a non trivial absolutely continuous part then $e_{p,n}(\rho)$ goes to 0 as $n^{-1/d}$.

Given $\rho$ in $\mathcal{P}_{p}$ the \emph{random empirical quantization error} of order $p$ is given by
\begin{equation*}
E_{p,n}(\rho):=W_{p}(R^{(n)},\rho),
\end{equation*}
where $R^{(n)}=n^{-1}\sum_{k=1}^{n} \delta_{X_{k}}$ is the empirical measure associated with the i.i.d. random variables $X_{k}$ which are distributed according to $\rho$. Let us mention here a result stated in \cite[Theorem 1]{fournier2014rate}: if $\rho$ is in $\mathcal{P}_{q}$ for some $q$ large enough then $\mathbb{E}\left[E_{p,n}(\rho)^{p} \right]^{1/p}$ goes to 0 as $n^{-1/2p}$ or $n^{-1/d}$ depending on the values of $p$ and $d$ (to be precise, an additional logarithmic term appears at the transition $p=d/2$). The rate $n^{-1/2p}$ comes from the fluctuations in the law of large numbers and the rate $n^{-1/d}$ comes from standard quantization as stated above.

Given $\rho$ in $\mathcal{P}_{p}$ the \emph{optimal deterministic empirical quantization error} of order $p$ is given by
\begin{equation*}
\tilde{e}_{p,n}(\rho):=inf_{\mu^{(n)}} W_{p}(\mu^{(n)},\rho),
\end{equation*}
where the infimum is taken with respect to the set of deterministic empirical measures $\mu^{(n)}$ supported by $n$ atoms. Up to our knowledge, the rate of convergence of $\tilde{e}_{p,n}(\rho)$ is known in dimension $d=1$ only and reads as follows.
\begin{thm}[\text{\cite[Theorem 5.20 and Remark 5.21]{xu2017best}}]\label{thm:xu:berger}
Let $p\geq 1$ and $d=1$.
\begin{enumerate}[(i)]
\item If $\rho\in\mathcal{P}_{q}$ with $q>p$ then $\tilde{e}_{p,n}(\rho)=o(n^{1/q-1/p})$.
\item If $\supp(\rho)$ is bounded then the rate of convergence of $\tilde{e}_{p,n}(\rho)$ is upper-bounded by $n^{-1/p}$. Furthermore, if the support of $\rho$ is disconnected then the rate $n^{-1/p}$ is sharp.
\end{enumerate}
\end{thm}

Combining the results of the standard quantization and deterministic empirical quantization we expect that for some generic $\rho$ with bounded support in dimension $d\geq 1$, the rate of $\tilde{e}_{p,n}(\rho)$ is given by $\max(n^{-1/d},n^{-1/p})$ (which is sharp when $\supp(\rho)$ is disconnected). This is what is shown in Theorem \ref{thm:rate:bounded:support} below (up to a logarithmic term at the transition $p=d$). Moreover, the generalization of Theorem \ref{thm:xu:berger}.$(i)$ to $d\geq 1$ is obtained in Corollary \ref{cor:rate:unbounded:support}.

\section{Main results}
\label{sec:main:result}

This section begins with a technical lemma which is used to control diameters in our construction of a uniform decomposition of probability measures (which is then given in Theorem \ref{thm:decomposition}).

\begin{lemma}\label{lem:construction:cube}
Let $r\geq 0$, $n\geq 1$ and $\nu$ be in $\mathcal{M}(\mathbb{R}^{d})$ with support included in $\mathcal{B}_{r}$ and total mass $|\nu|\geq 1/n$. There exists a subset $A$ of $\mathcal{B}_{r}$ such that $\nu(A)\geq 1/n$ and $\diam(A)\leq 4r(n|\nu|)^{-1/d}$.
\end{lemma}
\begin{proof}
Consider for any $r'\geq 0$ the maximal mass over balls of radius $r'$, namely
\begin{equation*}
m(r')=\sup \left\{ \nu(B(x,r')) :  x\in \mathbb{R}^{d} \right\}.
\end{equation*}
We prove by contradiction that
\begin{equation}
m(r \lfloor \left(n|\nu|\right)^{1/d} \rfloor^{-1})\geq 1/n.
\end{equation}
Assume that the $\nu$-mass of any ball of radius equal to $r \lfloor \left(n|\nu|\right)^{1/d} \rfloor^{-1}$ is less than $1/n$. Yet there exists a covering of the ball $[-r,r]^{d}$ into $ \lfloor \left(n|\nu|\right)^{1/d} \rfloor^{d}$ disjoint smaller balls, each one of radius equal to $r \lfloor \left(n|\nu|\right)^{1/d} \rfloor^{-1}$ (the balls are cubes). This implies 
\begin{equation*}
|\nu| < \lfloor \left(n|\nu|\right)^{1/d} \rfloor^{d} N^{-1} \leq ((n|\nu|)^{1/d})^{d} n^{-1}= |\nu|
\end{equation*}
yielding a contradiction.

Hence we have proved that we can find a subset $A$ such that $\nu(A)\geq 1/n$ and $\diam(A)\leq 2 r \lfloor \left(n|\nu|\right)^{1/d} \rfloor^{-1}$. The stated result then follows from
\begin{equation*}
r \lfloor \left(n|\nu|\right)^{1/d} \rfloor^{-1}\leq 2r (n|\nu|)^{-1/d}
\end{equation*}
(treat separately the two cases $\left(n|\nu|\right)^{1/d}\geq 2$ and $\left(n|\nu|\right)^{1/d}<2$).
\end{proof}

\begin{thm}\label{thm:decomposition}
Let $r\geq 0$ and $\rho$ be in $\mathcal{P}(\mathbb{R}^{d})$ with support included in $\mathcal{B}_{r}$. For all $n\geq 1$, there exist $\rho_{1},\dots,\rho_{n}$ in $\mathcal{M}(\mathbb{R}^{d})$ and $A_{1},\dots,A_{n}$ subsets of $\mathcal{B}_{r}$ such that $\rho=\sum_{k=1}^{n} \rho_{k}$ and 
\begin{equation*}
\forall	k=1,\dots,n, \quad |\rho_{k}|=\frac{1}{n},\ \supp(\rho_{k})\subset A_{k} \text{ and } \diam(A_{k}) \leq 4rk^{-1/d}.
\end{equation*}
\end{thm}
The proof is based on an iterative construction: each iteration relies on Lemma \ref{lem:construction:cube}.
\begin{proof}
Applying Lemma \ref{lem:construction:cube} to $\rho$ gives the existence of a subset $A_{n}$ such that $\rho(A_{n})\geq 1/n$ and $\diam(A_{n})\leq 4 r n^{-1/d}$. Then, we define the measure
\begin{equation*}
\rho_{n}:= \frac{n^{-1}}{\rho(A_{n})} \rho \mathbf{1}_{A_{n}}.
\end{equation*}
In particular, $|\rho_{n}|=1/n$ and $\supp(\rho_{n})\subset A_{n}$. Applying Lemma \ref{lem:construction:cube} to $\tilde{\rho}=\rho-\rho^{n}$ (its total mass is $(n-1)/n$) gives a subset $A_{n-1}$ such that $\rho(A_{n-1})\geq 1/n$ and $\diam(A_{n-1})\leq 4 r (n-1)^{-1/d}$. Similarly we define $\rho_{n-1}:= \frac{n^{-1}}{\tilde{\rho}(A_{n-1})} \tilde{\rho} \mathbf{1}_{A_{n-1}}$. Finally, applying $n$ times the iterative step ends the proof.
\end{proof}

The decomposition stated above is then used to control the rate of convergence of the optimal deterministic empirical quantization error $\tilde{e}_{p,n}(\rho)$ by exhibiting a particular empirical measure with controlled approximation error. The bounded case is treated in Theorem \ref{thm:rate:bounded:support}, the unbounded case in Corollary \ref{cor:rate:unbounded:support} and finally an application to the classification issue (when $\rho$ is an empirical measure) is given in Corollary \ref{cor:classification}.

\begin{thm}\label{thm:rate:bounded:support}
Let $r\geq 0$ and $\rho$ be in $\mathcal{P}(\mathbb{R}^{d})$ with support included in $\mathcal{B}_{r}$. For all $n\geq 1$, there exist $x_{1},\dots,x_{n}$ in $\mathbb{R}^{d}$, with associated empirical measure $\mu^{(n)}=n^{-1}\sum_{k=1}^{n}\delta_{x_{k}}$, such that for all $p\geq 1$,
\begin{equation*}
W_{p}(\mu^{(n)},\rho) \leq 4r f_{p,d}(n),
\end{equation*}
where
\begin{enumerate}[(i)]
\item if $p<d$, then $f_{p,d}(n) := (\frac{d}{d-p})^{1/p} n^{-1/d}$;
\item if $p=d$, then $f_{p,d}(n) := (\frac{1+\ln n}{n})^{1/d}$;
\item if $p>d$, then $f_{p,d}(n) :=  \zeta(p/d)n^{-1/p}$, where $\zeta$ is the Zeta Riemann function.
\end{enumerate}
\end{thm}

\begin{remark}
The rates show a transition between the rate for the approximation of a density $n^{-1/d}$ (standard quantization) and the rate for approximation of measures with disconnected support $n^{-1/p}$ (the simplest example being the sum of two Dirac masses, the interested reader is referred to \cite[Remark 5.21.(ii)]{xu2017best}). At the transition, our construction gives a rate with an additional logarithmic term. This may be an artefact of our too simple construction : this logarithmic term does not appear in dimension 1 for measures with bounded support - see \cite[Theorem 5.20.(ii)]{xu2017best}. However let us mention that such additional logarithmic term may appear for unbounded measures as highlighted in \cite[Example 5.8]{xu2017best}.
\end{remark}

\begin{proof}
Let $\rho_{1},\dots,\rho_{n}$ and $A_{1},\dots,A_{n}$ be respectively the measures and the subsets of $\mathcal{B}_{r}$ given by the decomposition of Theorem \ref{thm:decomposition}.
For each $k$, let $x_{k}$ denote the center of $A_{k}$ and let $\mu^{(n)}$ denote the associated empirical measure. We use the canonical coupling associated with the decomposition of $\rho$ into the $\rho_{k}$'s to control the Wasserstein distance. Namely,
\begin{equation*}
W_{p}(\mu^{(n)},\rho)^{p}\leq \sum_{k=1}^{n} W_{p}(n^{-1} \delta_{x_{k}},\rho_{k})^{p} \leq n^{-1} \sum_{k=1}^{n} \diam(A_{k})^{p} \leq \frac{(4r)^{p}}{n}\sum_{k=1}^{n} k^{-p/d}.
\end{equation*}
If $p>d$ then the sum is bounded by $\zeta(p/d)<+\infty$ and we obtain $(iii)$.
If $p<d$, then the sum is bounded by $\int_{0}^{n} t^{-p/d} dt = n^{1-p/d}/(1-p/d)$ which gives $(i)$.
If $p=d$, then the sum is bounded by $1+\ln n$ yielding $(ii)$.
\end{proof}

\begin{corollary}\label{cor:rate:unbounded:support}
Let $q\geq 1$ and $\rho\in \mathcal{P}_{q}$. For all $n\geq 1$, there exist $x_{1},\dots,x_{n}$ in $\mathbb{R}^{d}$, with associated empirical measure $\mu^{(n)}$, such that for all $p< q$,
\begin{equation*}
W_{p}(\mu^{(n)},\rho) = o(f_{p,d}(n)^{1-p/q})
\end{equation*}
where $f_{p,d}(n)$ is defined in Theorem \ref{thm:rate:bounded:support}.
\end{corollary}

\begin{proof}
We use a truncation argument to reduce to the case where $\rho$ is compactly supported.
Let $r>0$ be a truncation level to be chosen later and define the measure $\rho^{(r)}$ by
\begin{equation*}
\rho^{(r)}(dx):= \rho(dx)\mathbf{1}_{\mathcal{B}_{r}}(x) + \left(1-\rho(\mathcal{B}_{r})\right)\delta_{0}(dx).
\end{equation*}
By the canonical coupling, we have
\begin{equation*}
W_{p}(\rho,\rho^{(r)})^{p}\leq \int_{||x||>r} ||x||^{p} \rho(dx).
\end{equation*}
Yet, $\int_{||x||>r} ||x||^{p} \rho(dx)\leq C_{q}(r) r^{p-q}$ with $C_{q}(r):= \int_{||x||>r} ||x||^{q} \rho(dx)$ which goes to $0$ at $r\to +\infty$ by assumption. Without loss of generality one can replace $C_{q}(r)$ by some $C(r)$, satisfying $C(r)\geq 1/r$ and $\lim_{r\to +\infty} C(r)=0$, and write the upper-bound
\begin{equation}
W_{p}(\rho,\rho^{(r)})\leq C(r)^{1/p} r^{1-q/p}.
\end{equation}

By Theorem \ref{thm:rate:bounded:support}, for all $r\geq 0$, there exist empirical measures $\mu^{(n,r)}$ such that
\begin{equation*}
W_{p}(\mu^{(n,r)},\rho^{(r)})\leq 4r f_{p,d}(n).
\end{equation*}
By the triangular inequality,
\begin{equation*}
W_{p}(\mu^{(n,r)},\rho)\leq g(r) := C(r) r^{1-q/p} + 4r f_{p,d}(n).
\end{equation*}
To optimize $g(r)$, let us choose $\tilde{r}=\tilde{r}(n):=f_{p,d}(n)^{-p/q}$ since it satisfies $\tilde{r}^{1-q/p}=\tilde{r}f_{p,d}(n)$ and then consider $r(n):=C(\tilde{r})\tilde{r}$ to compute
\begin{equation*}
g(r(n))\leq C(C(\tilde{r})\tilde{r}) C(\tilde{r})^{1-q/p} f_{p,d}(n)^{1-p/q}+ 4 C(\tilde{r}) f_{p,d}(n)^{1-p/q}.
\end{equation*}
Finally, since $\lim_{n\to +\infty} \tilde{r}(n)=+\infty$ and $\lim_{r\to +\infty} C(r)=0$, we easily end the proof.
\end{proof}

\begin{corollary}
\label{cor:classification}
Assume that $N=cn$ with $c,n$ in $\mathbb{N}$. For any $x_{1},\dots,x_{N}$ in $\mathbb{R}^{d}$, there exist $C_{1},\dots,C_{n}$ disjoint subsets of indices of $\{1,\dots,N\}$ such that
\begin{itemize}
\item they form a uniform classification of $x_{1},\dots,x_{N}$, namely the cardinal $\card(C_{k})=c$ for all $k=1,\dots,n$;
\item each class is controlled, namely for all $k=1,\dots,n$,
\begin{equation*}
\diam(x_{(C_{k})})\leq 4rk^{-1/d},
\end{equation*}
where $x_{(C_{k})}:= \left\{ x_{i}, i\in C_{k} \right\}$ and $r=\max_{i=1,\dots,N}|x_{i}|$.
\end{itemize}
In particular, there exist $\overline{x}_{1},\dots,\overline{x}_{n}$ in $\mathbb{R}^{d}$ such that
\begin{equation}\label{eq:control:classification}
\frac{1}{N} \sum_{i=1}^{N} |x_{i}-\overline{x}_{k(i)}| \leq 4r f_{1,d}(n),
\end{equation}
where $k(i)\in \{1,\dots,n\}$ is such that $x_{i}\in C_{k(i)}$ and $f_{1,d}$ is given by Theorem \ref{thm:rate:bounded:support}.
\end{corollary}

\begin{proof}
The proof of the existence of the uniform classification $C_{1},\dots,C_{n}$ is based on an iterative application of Lemma \ref{lem:construction:cube} similar to the one developed in the proof of Theorem \ref{thm:decomposition} and is therefore omitted.

The proof of \eqref{eq:control:classification} is similar to the end of the proof of Theorem \ref{thm:rate:bounded:support}.
\end{proof}

\paragraph{Acknowledgments}
This research was supported by the project Labex MME-DII (ANR11-LBX-0023-01) and mainly conducted during the stay of the author at Universit\'e de Cergy-Pontoise.

\bibliographystyle{abbrv}
\bibliography{references}

\end{document}